\documentclass[twoside,11pt]{article}
\usepackage{amssymb}
\usepackage{amsfonts}
\usepackage{amscd}
\usepackage{amsmath}
\usepackage{amsthm}

\advance\topmargin by -\headheight
\advance\topmargin by -\headsep

\newcommand{\comment}[1]{}

\newcommand \al{\alpha}

\newcommand\ga{\gamma}
\newcommand\de{\delta}
\newcommand\ep{\varepsilon}
\newcommand\ze{\zeta}
\newcommand\et{\eta}
\renewcommand\th{\theta}

\newcommand\ka{\kappa}
\newcommand\la{\lambda}
\newcommand\rh{\rho}

\newcommand\De{\Delta}
\newcommand\Th{\Theta}

\newcommand\Om{\Omega}
\newcommand\Ph{\Phi}

\newcommand\ie{i.e.\ }

\newcommand\X{\mathfrak X}

\newcommand\on{\operatorname}
\newcommand\g{\mathfrak g}
\newcommand\h{\mathfrak h}

\renewcommand\k{\mathfrak k}

\newcommand\so{\mathfrak {so}}

\newcommand{\RR}{{\mathbb R}}

\newcommand{\x}{{\times}}
\newcommand{\pa}{{\partial}}

\newcommand{\oo}{\infty}
\renewcommand{\o}{\circ}
\newcommand\Ad{\on{Ad}}
\newcommand\ad{\on{ad}}

\newcommand\Diff{\on{Diff}}

\newcommand\grad{\on{grad}}

\renewcommand\div{\on{div}}

\newcommand\Den{\on{Den}}

\newcommand\C{{\mathcal C}}

\newcommand\F{{\mathcal F}}
\newcommand{\pg}{\Phi_{\ga}}
\newcommand{\pig}{\Phi^{-1}_{\ga}}
\newcommand{\dx}{\partial_x}
\newcommand{\dxg}{(\partial_x)_{\ga}}
\newcommand{\gd}{\dot{\ga}}
\newcommand{\gi}{\ga^{-1}}

\numberwithin{equation}{section}

\begin{document}

\newtheorem{theorem}{Theorem}[section]
\newtheorem{definition}[theorem]{Definition}
\newtheorem{lemma}[theorem]{Lemma}
\newtheorem{proposition}[theorem]{Proposition}
\newtheorem{corollary}[theorem]{Corollary}
\theoremstyle{remark}
\newtheorem{remark}[theorem]{Remark}
\theoremstyle{definition}
\newtheorem{example}[theorem]{Example}

\title{Generalized Euler-Poincar\'e equations on Lie groups and homogeneous spaces, orbit invariants and applications}

\author{Feride T\i\u{g}lay and Cornelia Vizman }
\maketitle

\begin{abstract}

We develop the necessary tools, including a notion of 
logarithmic derivative for curves in homogeneous spaces,
for deriving a general class of equations including Euler-Poincar\'e equations on Lie groups and homogeneous spaces. Orbit invariants 
play an important role in this context and we use these 
invariants to prove global existence and uniqueness 
results for a class of PDE. This class includes Euler-Poincar\'e equations that have not yet been considered in the literature as well as integrable 
equations like Camassa-Holm, Degasperis-Procesi,  $\mu$CH 
and $\mu$DP equations, and the geodesic equations 
with respect to right invariant Sobolev metrics on the group of 
diffeomorphisms of the circle. 

\end{abstract}




\section{Introduction}

The Euler-Poincar\'e equations on a Lie group $G$ are the  Euler-Lagrange equations for $G$--invariant Lagrangians $L:TG\to\RR$ \cite{MR}, \cite{MS}. They are written in the right invariant case
for the reduced Lagrangians $l$ on the Lie algebra $\g$ of $G$
as 
\begin{equation}\label{minus}
\frac{d}{dt}\frac{\de l}{\de u}=-\ad^*_u\frac{\de l}{\de u},
\end{equation}
while the left invariant case has no $-$ sign.
The Euler-Poincar\'e equations with quadratic Lagrangians are the geodesic equations on Lie groups 
with right invariant Riemannian metrics, called Euler equations, 
whose study was initiated by Arnold in the context of an incompressible ideal fluid \cite{Ar}.

We consider more general equations of the form 
\begin{equation}\label{teta}
\frac{d}{dt}\frac{\de l}{\de u}=-\th^*_u\frac{\de l}{\de u},
\end{equation}
where $\th^*$ is the infinitesimal action of $\g$ on $\g^*$ associated to a right group action $\Th^*$ of $G$ 
on $\g^*$, thus replacing the coadjoint actions. We call them generalized Euler-Poincar\'e equations.
This generalization is motivated by \cite{LMT}, 
where the action of the diffeomorphism group of the circle on $\la$--densities is considered,
and the Degasperis-Procesi equation and the $\mu$DP equation are obtained for $\la=3$. 
We consider the action of the diffeomorphism group of a manifold on 1-form $\la$--densities, 
thus obtaining a generalized EPDiff equation, which is the usual EPDiff equation for $\la=1$ 
and an extension of the Degasperis-Procesi equation to higher dimensions for $\la=2$.

The solutions of Euler-Poincar\'e equations preserve the coadjoint orbits.
This leads to an abstract Noether theorem that provides conserved quantities along the flow of the Euler-Poincar\'e equation \cite{HMR}.
We generalize this formalism and prove an abstract Noether theorem for the generalized Euler-Poincar\'e equations \eqref{teta}.

Geodesic equations with right $G$--invariant metrics on homogeneous spaces of $G$ are studied in \cite{KM} and \cite{LMP}, the main examples being the Hunter-Saxton equation and the multidimensional Hunter-Saxton equation, which are geodesic equations for right invariant $\dot H^1$ metrics on the homogeneous spaces 
$S^1\setminus\Diff(S^1)$ and $\Diff_\mu(M)\setminus\Diff(M)$ respectively.
These are Euler-Lagrange equations with quadratic Lagrangian: the energy.
We introduce Euler-Poincar\'e equations on the homogeneous space 
$H\setminus G$
as Euler-Lagrange equations with general right $G$--invariant Lagrangian functions.
For this purpose we introduce a right logarithmic derivative suited to homogeneous spaces of right cosets
$$
\bar\de^r:C^\oo(I,H\setminus G)\to C^\oo(I,\g)/C^\oo(I,H),\quad\bar\de^r\bar\ga=C^\oo(I,H)\cdot\de^r\ga,
$$ 
the logarithmic derivative of a curve $\bar\ga$ in $H\setminus G$
being the orbit of the logarithmic derivative of $\ga$, an arbitrary lift of $\bar\ga$.
The left action of $C^\oo(I,H)$, the path group of $H$, on $C^\oo(I,\g)$, the path space of the Lie algebra $\g$, 
is given here by $h\cdot u=\Ad(h)u+\de^r h$.
Consequently the Euler-Poincar\'e equations on homogeneous spaces
of right cosets are of the form \eqref{minus},
but the reduced Lagrangians $l$ on $\g$ have to be at the same time
$\h$--invariant under addition and $H$--invariant under the adjoint action.

The replacement of the coadjoint action $\ad^*$ by a Lie algebra action $\th^*$ of $\g$ on $\g^*$ produces generalized Euler-Poincar\'e equations on the homogeneous space $H\setminus G$ if the group $C^\oo(I,H)$ is a symmetry group of the equation \eqref{teta}.
This is accomplished if $\th^*$ is $H$--equivariant and if the restriction of $\th^*$ to the Lie algebra $\h$ of $H$ is the same as the restriction of the coadjoint action $\ad^*$ to $\h$.
The typical examples are the Lie algebra of vector fields on the circle acting on $\la$--densities and the Lie algebra of vector fields on a manifold with a volume form acting on 1-form $(\la-1)$-densities.
These lead us to the $\mu$Burgers equation \cite{LMT} and to
a multidimensional $\mu$Burgers equation.

We consider a class of nonlinear partial differential equations
\begin{equation}\label{eq:mom}
\partial_t \Phi u+u\partial_x \Phi u+\lambda (\partial_x u)\Phi u=0 
\end{equation}
where $\Phi=\sum_{j=0}^r (-1)^r \partial_x^{2r}$. These equations are generalized right Euler-Poincar\'{e} equations \eqref{teta} on the diffeomorphism group of the circle for the reduced Lagrangian $l(u)=\frac12\int_{S^1} u\Phi u dx$. The group coadjoint action $\Ad^*$ and the associated Lie algebra action $\ad^*$ are replaced by the group action $\Theta^*$ on $\lambda$--densities and the associated infinitesimal action $\theta^*$ respectively, so $\th^*_uf=u\pa_xf+\la (\pa_xu) f$. For $\lambda=2$ this is the coadjoint action, and \eqref{eq:mom} is the equation for geodesics on $\Diff(S^1)$ with respect to the right invariant metric defined by the $H^r$ inner product.  

Integrable equations such as Camassa-Holm and  Degasperis-Procesi  are notable members of this class of PDE. For $\Phi=\mu -\partial_x^2$ other integrable systems, namely $\mu$CH and $\mu$DP equations, are in the form \eqref{eq:mom}.
One also obtains all geodesic equations
with respect to right invariant Sobolev metrics on the group of 
diffeomorphisms of the circle.
Using orbit invariants we prove global (in time) existence and uniqueness of classical solutions of the periodic Cauchy problem for equation \eqref{eq:mom}.

\medskip

The paper is organized as follows.
In section 2 we present the general framework for Euler-Poincar\'e equations \eqref{minus}. Orbit invariants lead to an abstract Noether theorem of Holm, Marsden and Ratiu \cite{HMR}. Kelvin circulation theorems in ideal hydrodynamics and a Kelvin circulation theorem for the  EPDiff equation naturally follow from this abstract theorem. In order to set up the scene for the applications we also present explicit formulas for the orbit invariants for Camassa-Holm and $\mu$CH equations. 

In section 3 we develop a parallel framework for the generalized Euler-Poincar\'e equations \eqref{teta}. The orbit invariants for the underlying group action lead to a generalization of the abstract Noether theorem of \cite{HMR}. This generalization is motivated by the generalized EPDiff equations, which can be seen as higher dimensional versions of the family of equations derived by Degasperis and Procesi in \cite{DP}. A Kelvin circulation result for the generalized EPDiff equations follows.
Furthermore we give explicit formulas for orbit invariants for Degasperis-Procesi and $\mu$DP equations to be used in section 6.

In section 4 we introduce Euler-Poincar\'e equations on a homogeneous space as Euler-Lagrange equations for $G$--invariant Lagrangian functions $L$ on its tangent bundle.
First we introduce the logarithmic derivative for curves on
homogeneous spaces,
then we determine the invariance properties of the reduced Lagrangian functions $l$ on $\g$ and $\bar l$ on $\g/\h$.
For quadratic $l$ we recover the geodesic equations on homogeneous spaces from \cite{KM}.
Examples are the Hunter-Saxton equation, the multidimensional HS equation, as well as the Landau-Lifschitz equation.
We prove an abstract Noether theorem for Euler-Poincar\'e equations on homogeneous spaces.

Furthermore the coadjoint action $\ad^*$ can be replaced by another Lie algebra action $\th^*$ of $\g$ on $\g^*$ to produce generalized Euler-Poincar\'e equations on homogeneous spaces if $\th^*$ satisfies two extra conditions depending on the subgroup $H$ of $G$. 
This is shown in section 5, where a multidimensional
$\mu$Burgers equation is derived in this way.
An abstract Noether theorem for generalized Euler-Poincar\'e equations on homogeneous spaces is applied to the (multidimensional) 
$\mu$Burgers equation.

In section 6 we use orbit invariants presented in the previous sections to prove global (in time) existence and uniqueness of solutions to the periodic Cauchy problem for \eqref{eq:mom}. First we prove local well-posedness by restating the problem as an initial value problem for an ODE  on the appropriate infinite dimensional space and implementing a Banach contraction argument. Then we show that a sign condition on the initial data implies the estimates that guarantee the persistence of solutions for all time.


\section{Euler-Poincar\'e equations and orbit invariants}\label{s2}

In this section we present known results about orbit invariants and 
abstract Noether theorems for Euler-Poincar\'e equations,
including geodesic equations on Lie groups.
We add a few examples, most of them on the group of diffeomorphisms of the circle. For preliminaries on infinite dimensional Lie groups and their Lie algebras we refer to \cite{KM}, \cite{KW} or \cite{Neeb}.

Let $G$ be a Lie group and $\g$ its Lie algebra. 
The Euler-Lagrange equation for a right 
invariant La\-gran\-gian $L:TG\to\RR$, with value $l:\g\to\RR$ at the identity \cite{MR} \cite{MS}
\begin{equation}\label{EP}
\frac{d}{dt}\frac{\de l}{\de u}=-\ad^*_u\frac{\de l}{\de u}
\end{equation}
is called the {\it right Euler-Poincar\'e equation}.
Here $u=\ga'\ga^{-1}=\de^r\ga$ is a curve in $\g$, the right logarithmic derivative of the curve $\ga$ in $G$, so $\frac{\de l}{\de u}$ is a curve in $\g^*$. The Lie algebra coadjoint operator $\ad^*_\xi$ is given by the formula
$$
(\ad^*_\xi\al,\et)=(\al,\ad_\xi\et)\text{ for }\al\in\g^*
\text{ and }\xi,\et\in\g.
$$
With the classical notation $m=\frac{\de l}{\de u}$ for the momentum, 
the right Euler-Poincar\'e equation becomes 
\[
\tfrac{d}{dt}m=-\ad^*_um.
\]
For other approaches to the Euler-Poincar\'e equation we refer to \cite{H}.

The group coadjoint operator $\Ad^*_g$ is given by the formula
$$
(\Ad^*_g\al,\xi)=(\al,\Ad_{g}\xi),\quad\forall\xi\in\g.
$$
Note that $\Ad$ is a left action on $\g$, but $\Ad^*$ is a right action on $\g^*$.
It is straightforward to check the identity
\[
\frac{d}{dt}(\Ad^*_\ga m)=\Ad^*_\ga\Big(\frac{d}{dt}m+\ad^*_{\ga'\ga^{-1}}m\Big)
\]
for all curves $\ga$ in $G$ and all curves $m$ in $\g^*$.
This shows that a conserved quantity along solutions $u$ of the right Euler-Poincar\'e equation \eqref{EP} is
\begin{equation}\label{vval}
\Ad^*_\ga m= \Ad^*_\ga\frac{\de l}{\de u}= \text{const.}
\end{equation}
for any curve $\ga$ in $G$ satisfying $u=\ga'\ga^{-1}$. 
Similarly, $\Ad^*_{\ga^{-1}}m$ is a conserved quantity along solutions of the {\it left Euler-Poincar\'e equation} 
\begin{equation}\label{EPleft}
\frac{d}{dt}\frac{\de l}{\de u}=\ad^*_u\frac{\de l}{\de u},
\end{equation}
where $u=\ga^{-1}\ga'$ is the left logarithmic derivative of $\ga$.
In both situations the evolution of momentum $m=\frac{\de l}{\de u}$ 
takes place on a coadjoint orbit. Hence \eqref{vval} can be seen as  a formulation of the invariance of coadjoint orbits
under the right Euler-Poincar\'e flow. 

An abstract Noether theorem that formalizes the connection between invariance of coadjoint orbits and conserved quantities along the flow of the Euler-Poincar\'e equation is proved by Holm, Marsden and Ratiu in \cite{HMR}, 
and is designed to provide circulation type theorems. 
Let $\C$ be a $G$--manifold and $\ka:\C\to\g^{**}$ a $G$--equivariant map.
The Kelvin quantity $I:\C\x\g\to\RR$ is defined by
\begin{equation}\label{icu}
I( c ,u)=\Big(\ka( c ),\frac{\de l}{\de u}\Big).
\end{equation}
For $c$ driven by the right Euler-Poincar\'e flow, we can write the Kelvin quantity $I$ in terms of the coadjoint group action as
\begin{equation*}
I:=I(\ga. c _0,u)=\Big(\ga\cdot \ka(c_0),\frac{\de l}{\de u}\Big)=\Big(\ka(c_0), \Ad^*_{\ga}\frac{\de l}{\de u}\Big).
\end{equation*}
Then by \eqref{vval} the Kelvin quantity $I$ is conserved. This result is proved as an abstract Noether theorem for Euler-Poincar\'e equations in \cite{HMR}.

\begin{remark}\label{gval}
One can take $\C=\g$ and as $G$--equivariant map $\ka:\g\to\g^{**}$ simply the inclusion. With these assumptions the Kelvin quantity becomes $I(\xi,u)=(\frac{\de l}{\de u},\xi)$.
Then the abstract Noether theorem assures that along solutions of
the Euler-Poincar\'e equation \eqref{EP} the Kelvin quantity $I(t)=(\frac{\de l}{\de u},\Ad_{\ga(t)}\xi_0)$ is constant for any $\xi_0\in\g$.
\end{remark}

\begin{example}
In order to recover Kelvin circulation theorem in ideal hydrodynamics one takes $G=\Diff_\mu(M)$ the group of volume preserving diffeomorphisms on a compact Riemannian manifold $(M,g)$ with $\g=\X_\mu(M)$ the Lie algebra of divergence free vector fields and its regular dual $\g^*=\Om^1(M)/dC^\oo(M)$.
Considering $\C=\mathcal{L}$ the space of loops in $M$, the map 
\begin{equation*}
\ka:\mathcal{L}\to\g^{**},\quad (\ka( c ),[\al])=\int_ c \al,\quad [\al]\in\g^*,
\end{equation*}
is well defined and 
$G$--equivariant. 
The $L^2$ inner product on $\g$ defines a quadratic Lagrangian function $l(u)=\frac12\int_M g(u,u)\mu$.
The corresponding right Euler-Poincar\'e equation,
which is the geodesic equation for the right invariant $L^2$ metric,
is Euler equation for ideal fluid flow \cite{Ar} \cite{EM}
\begin{equation}\label{fluid}
\partial_tu+\nabla_uu=-\grad p,\quad\div u=0.
\end{equation}
Because $\frac{\de l}{\de u}=[u^\flat]$, the
Kelvin quantity \eqref{icu} associated to this map $\ka$ is $I=\int_c u^\flat$. The abstract Noether theorem implies that,
when the loop $ c$ is driven by the ideal fluid flow,
$I$ is conserved along solutions of \eqref{fluid}.
\end{example}

An abstract version of Kelvin-Noether theorem for Euler-Poincar\'e equations with advected parameters
is proved by Holm, Marsden and Ratiu in \cite{HMR}. This theorem is subsequently applied  to finite dimensional mechanical systems, as well as to continua: heavy top, compressible magneto-hydrodynamics and Maxwell fluid.  Furthermore applications to liquid crystals are introduced  by Holm in \cite{Holm}.

An abstract Noether theorem for semidirect products with applications to metamorphosis is presented by Holm and Tronci in \cite{HT}. The more general case of affine Lagrangian semidirect product theory is developed by Gay-Balmaz and Ratiu in \cite{GBR}, and the Kelvin-Noether theorem is adapted to this framework. The applications include a variety of equations gathered under the name complex fluids: spin systems, Yang-Mills and Hall magnetohydrodynamics, superfluids and microfluids.  

\paragraph{Geodesic equations.}
The previous example of ideal incompressible fluid was the one which motivated Arnold to 
develop an elegant geometric framework for studying differential equations \cite{Ar} \cite{AK}.
A special case of Euler-Poincar\'e equations of particular interest occurs for quadratic Lagrangians:
geodesic equations on Lie groups with right invariant Riemannian metric. 

In this case a symmetric operator $A:\g\to\g^*$, called the {\it inertia operator}, 
is employed to write the Lagrangian as $l(\xi)=\frac12(A\xi,\xi)$. 
This means that $l(\xi)=\frac12\langle \xi,\xi\rangle_\g$,
where the inner product on $\g$ is defined by $\langle\xi,\et\rangle_\g=(A\xi,\et)$. 
The inertia operator is injective if and only if the inner product on $\g$ is non-degenerate,
 in which case the Euler-Poincar\'e equation is the geodesic equation 
for the corresponding right invariant metric on the Lie group $G$.
The derivative
$\frac{\de l}{\de \xi}=A\xi=\langle \xi,\cdot\rangle_\g$ can be identified with $\xi$ because $A$ is injective, 
and the Euler-Poincar\'e equation \eqref{EP} is no other than the Euler equation
\begin{equation}\label{euler}
\tfrac{d}{dt}u=-\ad(u)^\top u.
\end{equation}
under this identification. Here $\ad(\xi)^\top$ is the adjoint of $\ad_\xi$ with respect to the inner product on $\g$, i.e.
$A(\ad(\xi)^\top\et)=\ad^*_\xi(A\et)$.

Denoting by $\Ad(g)^\top$ for $g\in G$, the adjoint of $\Ad_g$ with respect to the inner product,
we have an integrated version of the above identity:
$A(\Ad(g)^\top\et)=\Ad_g^*(A\et)$,
which can be used to show that the quantity $\Ad(\ga)^\top u$ is conserved 
along Euler equation \eqref{euler} (see also corollary 3.4 in \cite{MP}).

In the special case of geodesic equations on Lie groups with right invariant metrics \cite{Vizman}, Kelvin-Noether theorems were formulated for central Lie group extensions, semidirect product Lie groups, abelian Lie group extensions and a class of central extensions of semidirect products \cite{V7}.

\paragraph{The group of diffeomorphisms of the circle.}
The Lie algebra of the group of diffeomorphisms of the circle is $\X(S^1)$, the Lie algebra of vector fields on $S^1$ with the opposite Lie bracket.
We identify vector fields on the circle with functions
so that the Lie algebra adjoint action is given by $\ad_\xi\et=-[\xi,\et]=\xi'\et-\xi\et'$.
Under the identification of the regular dual of the Lie algebra $\X(S^1)$ 
with the space of functions via the $L^2$ inner product, the coadjoint action of the Lie algebra is
$\ad^*_\xi m=\xi m'+2\xi'm$.
Then the operator $\Ad_\ga$ that specifies the group adjoint action is 
\begin{equation}\label{adjo}
\Ad_\ga\xi=\partial_{\ep}|_{\ep=0} \left(\ga\circ\xi_{\ep}\circ\ga^{-1}\right)= \left(\xi\circ\ga^{-1}\right)\left(\ga'\circ\ga^{-1}\right),
\end{equation}
and the restriction of its adjoint to the regular dual is
$\Ad^*_\ga m=(m\o\ga)(\ga')^2$, the coadjoint action of $\Diff(S^1)$.
Note that in our convention the adjoint action is a left action of the diffeomorphism group,
while the coadjoint action is a right action of the diffeomorphism group.


\begin{example}\label{ex:B}
Burgers' equation
\begin{equation}\label{burgers2}
\partial_tu=-3uu'
\end{equation}
is the geodesic equation on the group of diffeomorphisms on the circle for the right invariant metric defined by the $L^2$ inner product $\langle u_1,u_2\rangle _{L^2}=\int_{S^1}u_1u_2dx$.
In this case $m=\frac{\de l}{\de u}$ is identified with $u$ and 
the invariant \eqref{vval} assures that, for $u$ solution of (\ref{burgers2}) and $\ga$ a corresponding geodesic, $\Ad^*_\ga u=(u\o\ga)(\ga')^2$ 
is conserved.
\end{example}

\begin{example}\label{ex:ch}
The Camassa-Holm equation \cite{CH}, \cite{FF}
\begin{equation}
\partial_t u-\partial_t u''+3uu'-uu'''-2u'u''=0
\end{equation}
is the geodesic equation on $\Diff(S^1)$
for the right invariant $H^1$ metric defined by the inner product
$\langle u_1,u_2\rangle _{H^1}=\int_{S^1}(u_1u_2+u_1'u_2')dx=\langle u_1,(1-\partial_x^2)u_2\rangle _{L^2}$. 
In this case $m=\frac{\de l}{\de u}=u-u''$ and the Camassa-Holm equation can be written in the form
\begin{equation}\label{ch2}
\partial_t m=-um'-2u'm.
\end{equation}
The orbit invariant \eqref{vval} assures that
\begin{equation*} 
\Ad^*_\ga m=(m\o\ga)(\ga')^2
\end{equation*}
is conserved, for a solution $u$ of (\ref{ch2}) and  a corresponding geodesic $\ga$.
\end{example}


\begin{example}\label{ex:much}
The $\mu$CH equation (it is introduced as $\mu$HS in \cite{KLM})
\begin{equation}\label{ghs}
\partial_tu''=2\mu(u)u'-2u'u''-uu''',
\end{equation}
where $\mu(u)=\int_{S^1}udx$ is the mean of the function $u$ on $S^1$, is the geodesic equation on the group of diffeomorphisms on the circle for the right invariant metric defined by the scalar product
$$\langle u_1,u_2\rangle _\mu=\int_{S^1}(\mu(u_1)\mu(u_2)+u_1'u_2')dx=\langle u_1,(\mu-\partial_x^2)u_2\rangle _{L^2}.$$ This equation admits a Lax pair and has a bihamiltonian structure \cite{KLM}.
By setting $m=\mu(u)-u''$ we can write it in the familiar form
\eqref{ch2}.
Then by \eqref{vval} we get like for the Camassa-Holm equation the conservation of 
\begin{equation*}
\Ad^*_\ga m=(m\o\ga)(\ga')^2,\quad m=\mu(u)-u''.
\end{equation*}
\end{example}

 
\paragraph{The EPDiff equation.}
A quadratic Lagrangian $l$ on the Lie algebra of vector fields $\X(M)$ on a compact Riemannian manifold $(M,g)$ can be expressed through a positive-definite symmetric operator $\Ph$ on $\X(M)$ by
\begin{equation}\label{quadratic}
l(u)=\frac12\int_M g(u, \Ph(u))\mu,
\end{equation}
where $\mu$ is the canonical volume form on $M$. 
The regular dual of $\X(M)$ is the space $\Om^1(M)\otimes\Den(M)$ of 1-form densities. 

The momentum density of the fluid is 
$m=\frac{\de l}{\de u}={\Ph(u)}^\flat \otimes\mu$, 
where the operator $\flat$ is associated to the Riemannian metric $g$. 
Since the coadjoint action is the Lie derivative, 
the right Euler-Poincar\'e equation on $\Diff(M)$ is $\partial_t m+L_um=0$. 
It is called the EPDiff equation, and
written for $\mathbf{m}=\Ph(u)$ in $\X( M)$ it takes the well known form \cite{HM}
\begin{equation}\label{eq:EPDiff}
\partial_t \mathbf{m}+u\cdot\nabla \mathbf{m}+(\nabla u)^\top\cdot \mathbf{m}+(\div u)\mathbf{m}=0.
\end{equation}
The EPDiff equation is the geodesic equation on $\Diff(M)$
with right invariant metric given by the symmetric operator $\Ph$.
In the special case $\mathbf{m}=\Ph(u)=u-\De u$, one obtains a higher dimensional Camassa-Holm equation: the geodesic equation 
on $\Diff(M)$ for the right invariant $H^1$ metric.
 
A known circulation result says that for a loop $ c $ in $ M$ and a density $\rho$ on $ M$, both driven by the EPDiff flow,
the quantity $\int_ c \frac{1}{\rho}\frac{\de l}{\de u}$ is conserved,
where $c=\ga\o c_0$ and $\rh=(\ga^{-1})^*\rh_0$ for $\ga'\ga^{-1}=u$ and fixed $c_0$ and $\rh_0$ \cite{HMR}.
The result fits in the setting of the abstract Noether theorem when choosing $\C=\mathcal{L}(M)\x\Den(M)$ the product of the space of loops and the space of densities on $M$, acted on by $\Diff(M)$. 
Then the equivariant map
\begin{equation}\label{equiv}
\ka:\mathcal{L}(M)\x\Den(M)\to\X( M)^{**},\quad (\ka(c,\rh),m)=\int_ c \frac{1}{\rh}m,\quad m\in\X(M)^*, 
\end{equation}
provides the Kelvin quantity $I=\int_c\frac{1}{\rho}\frac{\de l}{\de u}$. 
This means that  for a loop $c$ and a density $\rh=f\mu$, $f\in C^\oo(M)$, both driven by the EPDiff flow,
the quantity $I=\int_c\frac{1}{f}\mathbf{m}^\flat$ is constant along \eqref{eq:EPDiff}.


\section{Generalized Euler-Poincar\'e equations
and an abstract Noether theorem}\label{s3}

We consider a generalized Euler-Poincar\'e equation associated to a right invariant Lagrangian function $L:TG\to\RR$ with value $l:\g\to\RR$ at the identity and to a $G$--action $\Th$ on $\g$, namely 
\begin{equation}\label{gen}
\frac{d}{dt}\frac{\de l}{\de u}=-\th^*_u\frac{\de l}{\de u},
\end{equation}
where $u$ is the right logarithmic derivative of a curve $\ga$ in $G$, $\th$ is the infinitesimal action associated 
to the (left) group action $\Th$, and $\th_\xi^*$ is the adjoint of $\th_\xi$ for $\xi\in\g$. 
For $m=\frac{\de l}{\de u}$ the equation \eqref{gen} becomes $\frac{d}{dt}m=-\th_u^*m$. 

We notice that what is actually needed for the generalized Euler-Poincar\'e equations is just 
the (right) group action $\Th^*$ on $\g^*$ and its Lie algebra action $\th^*$.

\begin{proposition}\label{genact}
A conserved quantity along solutions $u$ of the generalized Euler-Poincar\'e equation \eqref{gen} is
$$
\Th^*_\ga(m)=\Th^*_\ga\Big(\frac{\de l}{\de u}\Big)=\text{const.}
$$ 
for any curve $\ga$ in $G$ satisfying $u=\ga'\ga^{-1}$.
\end{proposition}

An abstract Noether theorem holds for generalized Euler-Poincar\'e equations too.

\begin{theorem}\label{noet2}
Given a $G$--manifold $\C$ and a $G$--equivariant map $\ka:\C\to\g^{**}$, 
\ie it satisfies $\ka(\ga\cdot c)=\Th^{**}_{\ga}\ka(c)$ for all $c\in \C$, 
with $\Th^{**}_\ga$ the adjoint of $\Th^*_\ga$,
the Kelvin quantity $I(c,u)=\left(\ka(c),\frac{\de l}{\de u}\right)$ 
defined by $\ka$ is conserved for $u$ solution of the generalized right Euler-Poincar\'e equation \eqref{gen}, 
where $c=\ga\cdot c_0$, $c_0\in\C$, for $\ga$ a curve in $G$  with $u=\ga'\ga^{-1}$. 
\end{theorem}

\begin{proof}
We have
\[
I(\ga. c _0,u)=\Big(\Th_{\ga}^{**} \ka(c_0),\frac{\de l}{\de u}\Big)=\Big(\ka(c_0), \Th^*_{\ga}\frac{\de l}{\de u}\Big),
\]
hence the result follows from proposition \ref{genact}.
\end{proof}

\paragraph{Tensor densities.}

A tensor density of weight $\lambda\geq 0$ (respectively $\lambda<0$) 
on $S^{1}$ is a section of the bundle $\bigotimes^{\lambda} T^{*}S^{1}$ 
(respectively $\bigotimes^{-\lambda}TS^{1}$). We refer to \cite{gr}, \cite{ot} or \cite{o} for basic facts about the space of tensor densities. 
There is a well-defined (right) action of the diffeomorphism group 
$\mathrm{Diff}(S^1)$ on each density module 
\begin{equation*} 
\mathcal{F}_{\lambda} 
=
\left\{ mdx^{\lambda}: m\in C^{\infty}(S^{1}) \right\}.
\end{equation*}  
given by 
\begin{equation}  \label{3-action} 
\ga\cdot (m dx^\lambda) =
(m\circ\ga) \, (\ga')^\lambda dx^\lambda, 
\qquad 
\ga \in \mathrm{Diff}(S^1), 
\end{equation}
which naturally generalizes the coadjoint action 
$
\mathrm{Ad}^\ast$ 
on the regular dual of the Lie algebra $\X(S^1)$, identified with the space of quadratic differentials. 
The infinitesimal generator of the action in \eqref{3-action} 
is easily calculated,
\begin{equation} \label{Lie*} 
L_{u}^{\lambda}(m dx^{\lambda}) 
= 
\left( um'+\lambda u'm \right) dx^{\lambda}, 
\end{equation} 
and can be thought of as the Lie derivative of tensor densities. 
It represents the right action of $\X(S^1)$ on $\mathcal{F}_{\lambda}$.
The adjoint action \eqref{adjo}
 is the left action on $\mathcal{F}_{-1}$
and the 
coadjoint action is the right action on $\mathcal{F}_2$.
In general, the dual space to $\la$--densities is the space of $(1-\la)$--densities.

Replacing the coadjoint action
$\ad^*_um=um'+2u'm$, which is the action on 2-densities on the circle,
with the action on $\la$-densities, we write down the
generalized Euler-Poincar\'e equation on $\Diff(S^1)$ as
\[
\partial_tm=-um'-\la u'm.
\]
This approach of using the action of tensor densities has been introduced in \cite{LMT} to 
extend Arnold's geometric framework to include Degasperis-Procesi and $\mu$DP equations.


\begin{example}\label{ex:DP}
The Degasperis-Procesi equation \cite{DP}
\begin{equation}\label{DP}
\partial_t u -\partial_t u'' +4uu'-uu'''-3u'u''=0
\end{equation}
admits a Lax pair and a bihamiltonian structure.
It has a geometric interpretation 
on the space of tensor densities on the circle \cite{LMT}. 
Let $\Th$ be the left action of $\Diff(S^1)$ on $\mathcal{F}_{-2}$
and $\Th^*$ the right action of $\Diff(S^1)$ on $\mathcal{F}_3$ its dual. 
The corresponding generalized Euler-Poincar\'e equation on $\Diff(S^1)$ for the right invariant $H^1$ Lagrangian 
is the Degasperis-Procesi equation:
\[
 \partial_t m= -um'-3u'm,\quad m=u-u''.
\]
Applying proposition \ref{genact} we obtain the conserved quantity 
$$
\Th^*_\ga (m)=(m\o\ga)(\ga')^3, \quad m=u-u''
$$ 
for $\ga'\ga^{-1}=u$. This conserved quantity is observed both in \cite{EK} and \cite{LMT}.
\end{example}


\begin{example}\label{ex:muDP}
The $\mu$DP equation \cite{LMT}
\begin{equation}\label{muDP}
\mu(\partial_t u)-\partial_t u''+3\mu(u)u'-3u'u''-uu'''=0
\end{equation}
is an integrable equation with Lax pair formulation and bihamiltonian structure. It can be written in the form 
$$
\partial_t m=-um'-3u' m,\quad m=\mu(u)-u''.
$$ 

Like in example \ref{ex:DP} concerning the Degasperis-Procesi equation, we replace the coadjoint action $\ad^*_u$ by $\th^*_u = u\partial_x + 3u'$. Applying proposition \ref{genact} to the $\mu$DP equation, we obtain the conserved quantity 
\[ 
\Th^*_\ga (m)=(m\o\ga)(\ga')^3, \quad m=\mu(u)-u''
\]
for $\ga'\ga^{-1}=u$. 
This conservation is used in \cite{LMT} to prove a global existence theorem for the periodic Cauchy problem for $\mu$DP equation. 
\end{example}

 
\begin{example}\label{ex:gepdiff}
A generalized Euler-Poincar\'e equation on $\Diff(M)$, the diffeomorphism group of a Riemannian manifold
with canonical volume form $\mu$, can be obtained by considering 
the (right) $\Diff(M)$ action $\Th^*$ on $\Om^1(M)$
\begin{equation}\label{bigt}
\Th_\ga^*\al=J(\ga)^{\la-1}\ga^*\al,
\end{equation}
where  $J(\ga)$ denotes the Jacobian of the diffeomorphism $\ga$ with respect to the volume form $\mu$,
\ie $\ga^*\mu=J(\ga)\mu$. 
The infinitesimal action is $\th_u^*\al=L_u\al+(\la-1)(\div u)\al$.
For $\la=1$ it is the canonical action on 1-forms.
For $\la=2$ we recover the coadjoint action $\Ad^*_\ga\al=J(\ga)\ga^*\al$,
while identifying the space of 1-form densities, the regular dual of $\X( M)$, with $\Om^1( M)$.
In analogy to the case of tensor densities on the circle,
we say that the action \eqref{bigt} is the action of $\Diff(M)$ on the space 
$\Om^1(M)\otimes\F_{\la-1}(M)$ of 1-form $(\la-1)$--densities.

The generalized right Euler-Poincar\'e equation \eqref{gen} written for the Lagrangian $l$ 
given as in \eqref{quadratic} by a symmetric operator $\Ph$ on $\X( M)$ is
simply 
$$
\partial_t\al+L_u\al+(\la-1)(\div u)\al=0,\quad \al=\Ph(u)^\flat.
$$ 
We call this the {\it generalized EPDiff equation}. A more familiar form is
\begin{equation}\label{eq:gEPDiff}
\partial_t \mathbf{m}+u\cdot\nabla \mathbf{m}+(\nabla u)^\top\cdot \mathbf{m}+(\la-1)\mathbf{m}(\div u)=0,
\quad\mathbf{m}=\Ph(u). 
\end{equation}
In the special case $\la=3$ and $\mathbf{m}=u-\De u$ 
it extends the Degasperis-Procesi equation to higher dimensions.

Like for the EPDiff equation in section 2, there is a circulation result also for the generalized EPDiff equation. 
The group of diffeomorphisms acts on the loop space $\mathcal{L}(M)$ by $\ga\cdot c=\ga\o c$ and on the space $\F_{\la-1}(M)$ of $(\la-1)$--densities (identified with the space of smooth functions on $M$) from the left by 
$\ga^{-1}\cdot f=J(\ga)^{\la-1}(f\o\ga)$, $f\in C^\oo(M)$.
Considering the action $\Th^{**}$ on $\X( M)^{**}$, 
dual to the action $\Th^*$ in \eqref{bigt}, we obtain an equivariant map
\[
\ka:\mathcal{L}(M)\x\F_{\la-1}(M)\to\X( M)^{**},\quad (\ka(c,f),\al)=\int_ c \frac{1}{f}\al,\quad\al\in\Om^1(M), 
\]
generalizing the equivariant map \eqref{equiv}. Indeed,
\[
\left(\ka(\ga\cdot c,\ga\cdot f),\al\right)
=\int_{\ga\o c}\frac{1}{(f\o\ga^{-1})J(\ga^{-1})^{\la-1}}\al
=\int_c\frac{J(\ga)^{\la-1}}{f}\ga^*\al 
=\left(\ka(c,f),\Th^*_\ga\al\right).
\]

The associated Kelvin quantity is $I=\int_c\frac{1}{f}\frac{\de l}{\de u}$, with $\frac{\de l}{\de u}=\Ph(u)^\flat\in\Om^1(M)$.
Now the abstract Noether theorem \ref{noet2} ensures that for a loop $c$ in $ M$ and a $(\lambda-1)$--density $f$ on $ M$, both driven by the generalized EPDiff flow \eqref{eq:gEPDiff},
the Kelvin quantity $\int_ c \frac{1}{f}\mathbf{m}^\flat$ is conserved along \eqref{eq:gEPDiff}.

\end{example}


\section{Euler-Poincar\'e equations on homogeneous spaces}\label{s4}

In this paragraph we study the Euler-Lagrange equations for left invariant Lagrangians on the tangent bundle of a homogeneous space $G/H$ of left cosets. In a similar way a right invariant Lagrangian on the tangent bundle of a homogeneous space $H\setminus G$ of right cosets
can be treated.
The left Euler-Poincar\'e equations \eqref{EPleft} on a Lie group $G$  
are written in terms of left logarithmic derivatives 
$u=\ga^{-1}\ga'=\de^l\ga$ of curves in $G$.
The right logarithmic derivative is used in right Euler-Poincar\'e equations.
For writing Euler-Poincar\'e equations on homogeneous spaces we will need a kind of logarithmic derivative for curves in homogeneous spaces.

Smooth curves in $G/H$ can always be lifted to smooth curves in $G$, since $\pi:G\to G/H$ is a principal bundle.
Given a smooth curve $\bar\ga:I=[0,1]\to G/H$, we compare the left logarithmic derivatives of two smooth lifts $\ga,\ga_1:I\to G$ 
of $\bar\ga$, \ie $\bar\ga=\pi\o\ga=\pi\o\ga_1$.
There exists a smooth curve $h:I\to H$ such that $\ga_1=\ga h$,
hence 
$$
u_1=\de^l\ga_1=\de^l(\ga h)=h^{-1}\ga^{-1}(\ga' h+\ga h')
=\Ad(h^{-1}) u+\de^lh
$$
for $u=\de^l\ga:I\to\g$.
We notice that $u_1$ is obtained from $u$ via a right action of the group element $h\in C^\oo(I,H)$:
\begin{equation}\label{rightaction}
u\cdot h=\Ad(h^{-1}) u+\de^lh.
\end{equation}
It is a right action because of the identity
$\de^l(h_1h_2)=\Ad(h_2^{-1})\de^lh_1+\de^lh_2$.
This means one can define the {\it left logarithmic derivative} $\bar\de^l$ of a curve $\bar\ga$ in $G/H$ as an orbit under 
the right action \eqref{rightaction} of $C^\oo(I,H)$ on $C^\oo(I,\g)$, namely the orbit $ u\cdot C^\oo(I,H)$ of the left logarithmic derivative $ u$ of an arbitrary  lift $\ga:I\to G$ of $\bar\ga$, so
\[
\bar\de^l:C^\oo(I,G/H)\to C^\oo(I,\g)/C^\oo(I,H),\quad \bar\de^l\bar\ga=\de^l\ga\cdot C^\oo(I,H).
\]
When the subgroup $H$ is trivial, we recover the ordinary logarithmic derivative $\de^l$ for curves in $G$.

\begin{remark}\label{deright}
In the same way one defines a {\it right logarithmic derivative} $\bar\de^r$ for 
curves on the homogeneous space $H\setminus G$ of right cosets
\[
\bar\de^r:C^\oo(I,H\setminus G)\to C^\oo(I,\g)/C^\oo(I,H),
\quad \bar\de^r\bar\ga=C^\oo(I,H)\cdot\de^r\ga.
\]
where the group $C^\oo(I,H)$ acts on $C^\oo(I,\g)$ from the left by
\begin{equation}\label{leftaction}
h\cdot u=\Ad(h) u+\de^r h.
\end{equation}
\end{remark}

\begin{example}\label{diffss}
The rigid rotations of the circle form a subgroup $H$,
isomorphic to $S^1$, of the group $G=\Diff(S^1)$ 
of diffeomorphisms of the circle. The left action 
\eqref{leftaction} of
the group $C^\oo(I,S^1)$ on $C^\oo(I,\X(S^1))$ is
\begin{equation}\label{acte}
(a\cdot u)(t)(x)=u(t)(x-\tilde a(t))+\tilde a'(t),\quad t\in I,x\in\RR,
\end{equation}
where 
$\tilde a\in C^\oo(I,\RR)$ is any lift of the group element
$a\in C^\oo(I,S^1)$ and vector fields on $S^1$ are identified with periodic functions on $\RR$.
This action is involved in the definition of the right logarithmic derivative on the homogeneous space $S^1\setminus\Diff(S^1)$:
\[
\bar\de^r:C^\oo(I,S^1\setminus \Diff(S^1))\to 
C^\oo(I,\X(S^1))/C^\oo(I,S^1).
\]
\end{example}

The tangent bundle $TG$ of a Lie group $G$ carries a natural group multiplication, the tangent map of the group multiplication on $G$:
$\xi_g\cdot\et_a=g\et_a+\xi_ga$ for $\xi_g\in T_gG$ and $\et_a\in T_aG$.
Given a subgroup $H$ of $G$, its tangent bundle $TH$ is a subgroup of $TG$ and 
the submersion $T\pi:TG\to T(G/H)$ induces a diffeomorphism
between $TG/TH$ and $T(G/H)$.

A left $G$--invariant Lagrangian $\bar L:T(G/H)\to\RR$ 
determines a left $G$--invariant and right $TH$--invariant
Lagrangian $L=\bar L\o T\pi:TG\to\RR$,
because $T\pi$ is at the same time left $G$--equivariant and right
$TH$--invariant.
The left $G$--invariance and right $TH$--invariance of $L$
translates into $H$--invariance under adjoint action and $\h$--invariance under vector addition of 
its restriction $l:\g\to\RR$. The $\h$--invariance of $l$ is a direct consequence of the $TH$--invariance of $L$.
That $l$ is $\Ad(H)$--invariant follows from
\[
l(\Ad(h)\xi)=L(h\xi h^{-1})=L((h\xi)\cdot 0_{h^{-1}})
=L(h\xi)=l(\xi),
\]
for all $h\in H$ and $\xi\in\g$. 

The other way around, the left $G$--invariant $L$ defined by $l$
is also right $TH$--invariant. This follows from
\begin{align*}
L(\xi_g\cdot\ze_h)&=L(g\ze_h+\xi_gh)=l(h^{-1}\ze_h+h^{-1}g^{-1}\xi_gh)
=l(\Ad(h^{-1})(g^{-1}\xi_g))=l(g^{-1}\xi_g)\\
&=L(\xi_g),\text{ for all }\xi_g\in TG\text{ and }\ze_h\in TH, 
\end{align*}
using the $\h$--invariance of $l$ 
at step three and the $\Ad(H)$--invariance of $l$ at step four.

The left $G$--invariant Lagrangian $\bar L:T(G/H)\to\RR$ is uniquely determined by its restriction $\bar l$ to the tangent space $T_o(G/H)=\g/\h$ at $o=eH\in G/H$. 
Let $p:\g\to\g/\h$ denote the canonical projection.
Then $l=\bar l\o p$ is the restriction of $L=\bar L\o T\pi$.
The adjoint action of $H$ on $\g$ induces an adjoint action of $H$ on $\g/\h$.
The $\Ad(H)$--invariance of $l$ translates into an $\Ad(H)$--invariance of $\bar l$.

We summarize all these results in the next proposition.

\begin{proposition}
The following are equivalent data:
\begin{enumerate}
\item left $G$--invariant function $\bar L$ on $T(G/H)$;
\item right $TH$--invariant and left $G$--invariant function $L$ on $TG$;
\item $\h$--invariant and $\Ad(H)$--invariant function $l$ on $\g$;
\item $\Ad(H)$--invariant function $\bar l$ on $\g/\h$.
\end{enumerate}
\end{proposition}

The next result ensures that the Euler-Lagrange equations for left 
$G$--invariant Lagrangians on $T(G/H)$ 
look similar to left Euler-Poincar\'e equations.

\begin{theorem}\label{elhomo}
A solution of the Euler-Lagrange equation for a left 
$G$--invariant Lagrangian $\bar L:T(G/H)\to\RR$ 
is a curve $\bar\ga$ in $G/H$
such that the left logarithmic derivative $u=\ga^{-1}\ga'$
of a lift $\ga$ of $\bar\ga$ satisfies the left Euler-Poincar\'e equation
\begin{equation}\label{leftep}
\frac{d}{dt}\frac{\de l}{\de u}=\ad^*_u\frac{\de l}{\de u},
\end{equation}
for $l$ the ($\h$--invariant and $\Ad(H)$--invariant) restriction of $L=\bar L\o T\pi$ to $\g$.
\end{theorem}

\begin{proof}
A variation with fixed endpoints of the curve $\bar\ga$ in $G/H$
can be lifted to a variation with endpoints in $H$ of a lift $\ga$ in $G$ of $\bar\ga$, \ie $\bar\ga=\pi\o\ga$. Considering the Lagrangian $L=\bar L\o T\pi$ on $TG$ and denoting by $u=\ga^{-1}\ga'$ the left logarithmic derivative, we have:
\begin{multline*}
0=\de\int \bar L(\bar\ga'(t))dt
=\de\int L(\ga'(t))dt=\de\int l(u)dt
=\int\Big(\frac{\de l}{\de u},\de u\Big) dt\\
=\int\Big(\frac{\de l}{\de u},\frac{dv}{dt}+\ad_uv\Big) dt
=\int\frac{d}{dt}\Big(\frac{\de l}{\de u},v\Big) dt
-\int\Big(\frac{d}{dt}\frac{\de l}{\de u},v\Big) dt\\
+\int\Big(\frac{\de l}{\de u},\ad_uv\Big) dt
=\Big(\frac{\de l}{\de u},v_1\Big)
-\Big(\frac{\de l}{\de u},v_0\Big)
+\int\Big(-\frac{d}{dt}\frac{\de l}{\de u}+\ad^*_u\frac{\de l}{\de u},v\Big) dt.
\end{multline*}
Here $v$ denotes the left logarithmic derivative of the variation 
of $\ga$, so $\de u=\frac{dv}{dt}+[u,v]$.
The variation of $\ga$ has endpoints in $H$, so $v_0,v_1\in\h$.
The $\h$--invariance of $l$ ensures that $\frac{\de l}{\de u}$ vanishes on $\h$, hence the previous calculation gives
\[
\int\Big(\frac{d}{dt}\frac{\de l}{\de u}-\ad^*_u\frac{\de l}{\de u},v\Big) dt=0.
\]
It follows that $u$ satisfies the Euler-Poincar\'e equation \eqref{leftep}.

It is easy to verify that if the Euler-Poincar\'e equation is
satisfied by the logarithmic derivative of one lift of $\bar\ga$, it is satisfied 
by the logarithmic derivative of any lift of $\bar\ga$. Let $\ga$ and $\ga_1=\ga h$ be two such lifts and let $u$ and $u_1$ be their left logarithmic derivatives. Then
\begin{align*}
\frac{d}{dt}\frac{\de l}{\de u_1}-\ad^*_{u_1}\frac{\de l}{\de u_1}
&=\frac{d}{dt}\Big(\Ad^*_h\frac{\de l}{\de u} \Big)
-\ad^*_{(\Ad(h^{-1})u+\de^lh)}\Big(\Ad^*_h\frac{\de l}{\de u}\Big) \\
&=\Ad^*_h\Big(\frac{d}{dt}\frac{\de l}{\de u}\Big) 
-\ad^*_{(\Ad(h^{-1})u)}\Big(\Ad^*_h\frac{\de l}{\de u}\Big) \\
&=\Ad^*_h\Big(\frac{d}{dt} \frac{\de l}{\de u} -\ad^*_u\frac{\de l}{\de u}\Big).
\end{align*}
because $u_1=\Ad(h^{-1})u+\de^lh$ and
$\frac{\de l}{\de u_1} =\frac{\de l}{\de u} \o\Ad_h=\Ad^*_h\frac{\de l}{\de u} $.
\end{proof}

We call \eqref{leftep} the left Euler-Poincar\'e equation on the homogeneous manifold $G/H$. 
One has an orbit invariant for this equation, 
similar to \eqref{vval}. 

\begin{proposition}
The quantity $\Ad^*_{\ga^{-1}}\frac{\de l}{\de u}\in\g^*$
is conserved along the left Euler Poincar\'e equation \eqref{leftep} on $G/H$ with 
$\h$--invariant and $\Ad(H)$--invariant Lagrangian function $l$ on $\g$,
where $\ga$ is any lift of $\bar\ga$ and $u=\de^l\ga$.
\end{proposition}

The independence on the choice of the lift $\ga$ is immediate: for $\ga_1=\ga h$,
\[
\Ad^*_{\ga_1^{-1}}\frac{\de l}{\de u_1}=\Ad^*_{(\ga h)^{-1}}\Ad^*_h\frac{\de l}{\de u}=\Ad^*_{\ga^{-1}}\frac{\de l}{\de u}.
\] 


The abstract Noether theorem written for left Euler-Poincar\'e equations on Lie groups ensures that the Kelvin quantity \eqref{icu}
is conserved for $u$ solution of \eqref{EPleft}.
One can formulate an analogous abstract Noether theorem for homogeneous spaces.

\begin{theorem}
Considering a $G$--manifold $\C$ and a map 
$\ka:\C\to\g^{**}$ which is $G$--equivariant, the Kelvin quantity 
\begin{equation}\label{icun}
I:\C\x\g\to\RR,\quad I( c ,u)=\Big(\ka( c ),\frac{\de l}{\de u}\Big)
\end{equation}
is conserved along solutions $\bar\ga$ of the Euler-Lagrange equation on $G/H$ with left invariant Lagrangian $\bar L$, namely for $\ga$ a curve in $G$ lifting $\bar\ga$, $u$ its left logarithmic derivative and $ c =\ga^{-1}\cdot c _0$, $ c _0\in\C$. 
\end{theorem}

\begin{proof}
This follows from the abstract Noether theorem on $G$ mentioned above.
In addition we verify that the Kelvin quantity does not depend 
on the choice of the lift $\ga$ of $\bar\ga$:
\begin{align*}
I( c _1,u_1)&=\Big(\ka(\ga_1^{-1}\cdot c _0),\frac{\de l}{\de u_1}\Big)
=\Big(\ka(h^{-1}\ga^{-1}\cdot c _0),\Ad^*_h\frac{\de l}{\de u} \Big)\\
&=\Big(\ka(\ga^{-1}\cdot c _0),\frac{\de l}{\de u} \Big)=I( c ,u), 
\end{align*}
where $\ga_1=\ga h$, $c=\ga^{-1}\cdot c _0$ and $c _1=\ga_1^{-1}\cdot c _0$.
\end{proof}


\paragraph{Geodesic equations on homogeneous spaces.}
A special Euler-Lagrange equation is the Euler equation on homogeneous spaces: geodesic equation 
for a left $G$--invariant Riemannian metric on $G/H$. 
It was  studied by Khesin and Misiolek in \cite{KM}, using the Hamiltonian point of view. 

Let $A:\g\to\g^*$ be a symmetric degenerate (inertia) operator with kernel $\h$, such that $A$ is $H$--equivariant. 
We consider the Lagrangian $l(\xi)=\frac12(A\xi,\xi)$ on $\g$.
Taking into account the symmetry of $A$, the $\h$--invariance of $l$ is easily checked: $l(\xi+\ze)=\frac12(A\xi+A\ze,\xi+\ze)=\frac12(A\xi,\xi)+\frac12(A\xi,\ze)
=l(\xi)+\frac12(\xi,A\ze)=l(\xi)$, 
for all $\xi\in\g$ and $\ze\in\h$. 
The $H$--equivariance of $A$ ensures the $\Ad(H)$--invariance of $l$,
so $l$ descends to an $\Ad(H)$--invariant Lagrangian $\bar l:\g/\h\to\RR$, as needed for the Euler-Poincar\'e equation 
\eqref{leftep} on homogeneous spaces.

Now $m=\frac{\de l}{\de u}=Au$, so the left Euler-Poincar\'e equation
on $G/H$ writes $\frac{d}{dt}Au=\ad^*_u(Au)$.
It is the image under the inertia operator $A$ of 
the Euler equation, the left invariant version of the Euler equation \eqref{euler}: 
$$
\frac{d}{dt}u=\ad(u)^\top u.
$$
This can be interpreted as the geodesic equation 
on $G/H$ for the left invariant Riemannian metric
coming from the degenerate inner product
$\langle\xi,\et\rangle=(A\xi,\et)$ on $\g$ \cite{KM}.

\begin{example}\label{ex:HS}
The Hunter-Saxton equation describing weakly nonlinear unidirectional waves \cite{HS}
\begin{equation}\label{hs}
\partial_tu''=-2u'u''-uu'''
\end{equation}
is a geodesic equation on the homogeneous space $S^1\setminus \Diff(S^1)$ of right cosets with the right invariant metric defined by the degenerate $\dot H^1$ inner product $\langle u_1,u_2\rangle=\int_{S^1}u'_1u'_2dx$ on $\X(S^1)$ \cite{KM}. It fits into the framework above when $A(u)=-u''$. The two conditions are easily verified: the kernel of $A$ is $\RR$, the Lie algebra of the subgroup of rigid rotations, and $A$ is $S^1$--equivariant.

In this case $l(u)=\frac12\langle u,u\rangle=\frac12\int_{S^1}(u')^2dx$,
so $m=\frac{\de l}{\de u}=-u''$ satisfies 
$$
\partial_tm=-um'-2u'm,
$$
which gives Hunter-Saxton equation \eqref{hs}.
It has to be read as an equation for the $C^\oo(I,S^1)$-orbit $\bar u=C^\oo(I,S^1)\cdot u$ of $u\in C^\oo(I,\X(S^1))$ under the left action \eqref{acte},
which plays the role of the right logarithmic derivative of a curve
$\bar\ga:I\to S^1\setminus\Diff(S^1)$.

A conserved quantity for the Hunter-Saxton equation is
$$
\Ad^*_\ga m
=-(u''\o\ga)(\ga')^2,
$$ 
where $\ga:I\to\Diff(S^1)$ is any lift of the curve $\bar\ga$.
\end{example}

\begin{example}\label{higherHS}
This example concerns the multidimensional Hunter-Saxton equation
from \cite{LMP}. 
Let $M$ be a compact manifold and let $\mu$ be a fixed volume form on $M$. We consider the homogeneous space  $\Diff_\mu(M)\setminus\Diff(M)$ of right cosets, where  $\Diff_\mu(M)$ is the subgroup of volume preserving diffeomorphisms of $M$.
The Lagrangian 
\[
l:\X(M)\to\RR,\quad l(u)=\frac12\int_M(\div u)^2\mu
\]
is both $\X_\mu(M)$--invariant and $\Ad(\Diff_\mu(M))$--invariant,
so we have a corresponding Euler-Poincar\'e equation
on the homogeneous space  $\Diff_\mu(M)\setminus\Diff(M)$:
\[
\pa_tm=-L_um,\quad m=\frac{\de l}{\de u}=-d(\div u)\mu\in\X(M)_{reg}^*
=\Om^1(M)\otimes\Den(M).
\]
It has a similar expression to Euler-Poincar\'e equation on $\Diff(M)$: the EPDiff equation in section \ref{s2}.

We replace the special form of the momentum $m$ and we drop the constant density $\mu$ to obtain the following equation in $\Om^1(M)$
\[
\pa_t d(\div u)=-d L_u (\div u)-(\div u) d(\div u).
\]
It coincides with the Hunter-Saxton equation when $M=S^1$: the subgroup of volume preserving diffeomorphisms of the circle is isomorphic to the subgroup of rigid rotations of the circle.
\end{example}

\begin{example}\label{LL}
Let $K$ be a Lie group with Lie algebra $\k$ possessing a $K$--invariant inner product $\langle\ ,\ \rangle_\k$.
The Lie algebra of the loop group $LK:=C^\oo(S^1,K)$ is the loop algebra $L\k=C^\oo(S^1,\k)$.
The subgroup of constant loops, identified with $K$,
defines the homogeneous space of right cosets $K\setminus LK$.

Each $\k$--invariant and $\Ad(K)$--invariant Lagrangian $l$ on $L\k$
determines a right Euler-Poincar\'e equation on $K\setminus LK$:
\begin{equation}\label{kapa}
\pa_tm=[u,m],\quad m=\frac{\de l}{\de u}. 
\end{equation}
Here $m$ is a curve in $L\k$,
since the inner product $\langle\ ,\ \rangle_\k$ permits the identification of the regular dual
of $L\k$ with $L\k$.
If one considers the Lagrangian defined by the $\dot H^1$ inner product: 
\[
l(u)=\frac12\int_{S^1}\langle u',u'\rangle_{\k} dx,
\]
then $m=-u''$ and the Euler-Poincar\'e equation \eqref{kapa} becomes 
$\pa_tu''=[u,u'']$.
Another possibility would be the Lagrangian 
\[
l(u)=\frac12\int_{S^1}\langle u,u\rangle_{\k} dx-\frac12\langle\mu(u),\mu(u)\rangle_{\k}
=\frac12\int_{S^1}\langle u-\mu(u),u-\mu(u)\rangle_{\k} dx.
\]
This time $m=u-\mu(u)$, so the equation \eqref{kapa} is
$\pa_tu-\mu(\pa_tu)=-[u,\mu(u)]$.

More important is the $\dot H^{-1}$ Lagrangian
\[
l(u)=\frac12\int_{S^1}\langle\partial_x^{-1}u,\partial_x^{-1}u\rangle_{\k}dx
=-\frac12\int_{S^1}\langle\partial_x^{-2}u,u\rangle_{\k}dx
\]
because it leads to the
Landau-Lifschitz equation \cite{AK} \cite{Kambe}.
It is well defined on the homogeneous space $\k\setminus L\k$, 
which can be identified with the space of all derivatives of loops in $\k$.
It is also $\Ad(K)$--invariant,
so it fits well into our setting for Euler-Poincar\'e equations on homogeneous spaces.
We get $m=-\pa_x^{-2}u$, so $u=-m''$, and the equation \eqref{kapa} 
becomes
$\partial_t m=[m,m'']$.
In the special case $K=SO(3)$ we get the 
Landau-Lifschitz equation 
\[
\pa_tL=L\x L'',
\]
where one identifies the Lie algebras $(\so(3),[\ ,\ ])$
and $(\RR^3,\x)$.
This equation is equivalent to the vortex filament equation
$\pa_t c=c'\x c''$, for $L=c'$ the tangent vector to the filament, 
a closed arc-parametrized time-dependent curve $c$ in $\RR^3$.
\end{example}


\section{Generalized Euler-Poincar\'e equations on homogeneous spaces}

The coadjoint action $\ad^*$ in the right Euler-Poincar\'e equation on the homogeneous space $H\setminus G$ of right cosets can be replaced with another action $\th^*$ to give
a generalized right Euler-Poincar\'e equation on $H\setminus G$,
similarly to the generalized Euler-Poincar\'e equation on Lie groups from section \ref{s3}.
This time we have to impose some conditions on $\th^*$, so that the following holds: 
if the right logarithmic derivative $u=\de^r\ga$ of one lift of the curve $\bar\ga$ satisfies the generalized Euler-Poincar\'e equation
$\frac{d}{dt}\frac{\de l}{\de u}=-\th^*_u\frac{\de l}{\de u}$,
then the right logarithmic derivatives of all the other lifts of $\bar\ga$ satisfy the same equation.
In other words the group $C^\oo(I,H)$ with the left action \eqref{leftaction} has to be a symmetry group of the above generalized Euler-Poincar\'e equation. 

\begin{proposition}\label{condi}
Let $H$ be a subgroup of $G$ with Lie algebra $\h$, and let $\th^*$ be a Lie algebra action of $\g$ on $\g^*$.
If the map $\th^*:\g\x\g^*\to\g^*$ is $H$--equivariant and if the action $\th^*$ restricted to $\h$ 
equals the coadjoint action $\ad^*$ restricted to $\h$,
then $C^\oo(I,H)$ is a symmetry group of the equation 
$\frac{d}{dt}\frac{\de l}{\de u}=-\th^*_u\frac{\de l}{\de u}$,
for the left action \eqref{leftaction}.
\end{proposition}

\begin{proof}
We have to show that for any solution $u\in C^\oo(I,\g)$ of the generalized right Euler-Poincar\'e equation, and for any $h\in C^\oo(I,H)$,
the curve $h\cdot u=\Ad(h)u+\de^rh$ is again a solution of the generalized right Euler-Poincar\'e equation. 
In the computation below, we will use the fact 
that $h$ acts on $m=\frac{\de l}{\de u}$ by the coadjoint action:
$h\cdot m=\Ad^*_{h^{-1}}m$. 

The $H$--equivariance of $\th^*$ means that
$\th^*_{\Ad(h)u}(\Ad^*_{h^{-1}}m)=\Ad^*_{h^{-1}}\th^*_um$.
Knowing also that $\th^*_{\de^rh}=\ad^*_{\de^rh}$ for any curve 
$h\in C^\oo(I,H)$, we compute
\begin{gather*}
\pa_t(h\cdot m)+\th^*_{h\cdot u}(h\cdot m)
=\partial_t(\Ad^*_{h^{-1}}m)+\th^*_{\Ad(h)u+\de^rh}(\Ad^*_{h^{-1}}m)\\
=\Ad^*_{h^{-1}}(\partial_t m)-\ad^*_{\de^rh}(\Ad^*_{h^{-1}}m)
+\th^*_{\de^rh}(\Ad^*_{h^{-1}}m)
+\th^*_{\Ad(h)u}(\Ad^*_{h^{-1}}m)\\
=\Ad^*_{h^{-1}}(\partial_t m+\th^*_um).
\end{gather*}
This shows the symmetry of the generalized right Euler-Poincar\'e equation under the group $C^\oo(I,H)$.
\end{proof}

This proposition ensures that if the two conditions on $\th^*$ 
are satisfied, then the generalized right Euler-Poincar\'e equation is an equation for curves in the homogeneous space. 

\begin{remark}\label{410}
The orbit invariant $\Th_\ga^*\frac{\de l}{\de u}$ from proposition \ref{genact} for the generalized right Euler-Poincar\'e equation on Lie groups is also an orbit invariant for the generalized right Euler-Poincar\'e equation on homogeneous spaces.
\end{remark}

\paragraph{Tensor densities.}
When $G=\Diff(S^1)$ and $H=S^1$, 
the action of the Lie algebra $\X(S^1)$ on $\la$--densities,
\ie $\th^*_um=um'+\la u'm$,
satisfies the two conditions required in proposition \ref{condi}.

The first condition on $\th^*$, its $H$--equivariance, can be verified infinitesimally since $S^1$ is connected:
\[
\th^*_{\ad_wu}m-\th^*_u\ad^*_wm=-\ad^*_w\th^*_um,\quad\forall u\in\g,m\in\g^*,w\in\h.
\] 
Both sides of the equality give $wu'm'+wum''+\la wu''m+\la wu'm'$, since $w'=0$.
The second condition on $\th^*$ is easily verified: $\th^*_wm=wm'=\ad^*_wm$ for all $w\in\h=\RR$.

This means we can replace the coadjoint action
$\ad^*_um=um'+2u'm$, which is the action on 2-densities on the circle,
with the action on $\la$-densities, to write down a 
generalized Euler-Poincar\'e equation on 
the homogeneous space $S^1\setminus\Diff(S^1)$:
\[
\partial_tm=-um'-\la u'm.
\]

The same thing can be done in higher dimensions too.
Let $G=\Diff(M)$ be the diffeomorphism group of a connected compact manifold,
and $H=\Diff_\mu(M)$ the subgroup of volume preserving diffeomorphisms, 
where $\mu$ is a fixed volume form on $M$. 
As in example \ref{ex:gepdiff}, 
we identify $\Om^1(M)$ with the regular dual of $\X( M)$ 
using the volume form, so the coadjoint action can be written as $\ad^*_u\al=L_u\al+(\div u)\al$, 
an action on 1-form densities. The action $\th^*$ on 1-form $(\la-1)$--densities: 
\begin{equation}\label{acla}
\th_u^*\al=L_u\al+(\la-1)(\div u)\al
\end{equation}
satisfies the conditions required in proposition \ref{condi}, 
hence it provides a generalized Euler-Poincar\'e equation on 
the homogeneous space $\Diff_\mu(M)\setminus\Diff(M)$:
$$
\partial_t\al=-L_u\al-(\la-1)(\div u)\al=0.
$$ 

The first condition on $\th^*$ can be verified as follows:
\begin{align*}
\th^*_{\Ad_{h^{-1}}u}\Ad^*_h\al
&=\th^*_{h^*u}(h^*\al)
=L_{h^*u}(h^*\al)+(\la-1)(\div h^*u)h^*\al\\
&=h^*(L_u\al+(\la-1)(\div u)\al)
=h^*(\th^*_u\al)
=\Ad^*_h(\th^*_u\al),
\end{align*}
using the fact that $J(h)=1$ and $\div(h^*u)=h^*\div u$ for all $h\in\Diff_\mu(M)$.
The second condition follows from $\th^*_w\al=L_w\al=\ad^*_w\al$
for all $w\in\X_\mu(M)$.

\begin{example}\label{ex:muB}
The periodic $\mu$Burgers equation 
\begin{equation}\label{mubu}
-\partial_t u''-3u'u''-uu'''=0
\end{equation}
is shown to admit a Lax pair formulation  and a bihamiltonian structure in \cite{LMT}. 
This terminology is related to a reformulation of this equation as 
$(\pa_tu+uu')'=0$, hence as $\pa_tu+uu'=-\mu(\pa_tu)$,
where $\mu$ denotes the mean of a function on the circle.

From proposition \ref{condi} follows that,
by replacing the coadjoint action by the action of $\X(S^1)$
on tensor densities on the circle, the family of equations
\[
\partial_tm=-um'-\la u'm,\quad m=-u''
\]
can be interpreted as generalized right Euler-Poincar\'e equations on the homogeneous space $S^1\setminus\Diff(S^1)$. The Lagrangian 
is given here by the $\dot H^1$ inner product:
\begin{equation}\label{lag}
l(u)=\frac12\int_{S^1}(u')^2dx. 
\end{equation}
For $\la=2$ one obtains a geodesic equation: the Hunter-Saxton equation from example \ref{ex:HS}. For $\la=3$ one obtains
a generalized Euler-Poincar\'e equation: the $\mu$Burgers equation \eqref{mubu}.

Applying remark \ref{410} to the 
$\mu$Burgers equation, we get the conserved quantity 
$$
\Th^*_\ga(m)=(m\o\ga)(\ga')^3,\quad m=-u'', 
$$
for any lift $\ga:I\to\Diff(S^1)$ of the solution curve $\bar\ga:I\to S^1\setminus\Diff(S^1)$ whose right logarithmic derivative is $\bar u=C^\oo(I,S^1)\cdot u$, \ie $\ga'\circ\ga^{-1}=u$.
\end{example}

\begin{example}\label{ex:multimuB}
There is a multidimensional $\mu$Burgers equation
which can be obtained as a generalized Euler-Poincar\'e equation.
On one hand, as observed for  
the multidimensional Hunter-Saxton equation
from example \ref{higherHS}, the Lagrangian  
\[
l:\X(M)\to\RR,\quad l(u)=\frac12\int_M(\div u)^2\mu
\]
is both $\X_\mu(M)$--invariant and $\Ad(\Diff_\mu(M))$--invariant.
On the other hand, we already showed that the action \eqref{acla} of the Lie algebra $\X(M)$ on 1-form $(\la-1)$-densities
is $\Diff_\mu(M)$--equivariant and its restriction to $\X_\mu(M)$ equals the restriction to $\X_\mu(M)$ of the coadjoint action.
Hence, by proposition \ref{condi}, we have a corresponding generalized Euler-Poincar\'e equation
on the homogeneous space  $\Diff_\mu(M)\setminus\Diff(M)$:
\begin{equation}\label{mdmb}
\pa_t d(\div u)=-d L_u (\div u)-(\la-1)(\div u) d(\div u).
\end{equation}

For $\la=3$, this generalized Euler-Poincar\'e equation 
is the multidimensional $\mu$Burgers equation. It can be rewritten as
\[
d\div(\pa_tu+(\div u)u)=0,
\]
because of the identity $\div((\div u)u)=L_u\div u+(\div u)^2$.
We deduce that the time dependent function $\div(\pa_tu+(\div u)u)$
is constant on $M$. Its integral over $M$ vanishes, so it is the zero function.
Thus we have a simpler expression for the 
multidimensional $\mu$Burgers equation \eqref{mdmb}:
\begin{equation}\label{scurt}
\div(\pa_tu+(\div u)u)=0,
\end{equation}
thus generalizing $(\pa_tu+u'u)'=0$, the $\mu$Burgers equation 
on the circle.

Applying remark \ref{410} to the multidimensional
$\mu$Burgers equation, we get a conserved quantity 
along \eqref{scurt}:
$$
\Th_\ga^*\al=J(\ga)^2\ga^*\al,\quad \al=-d(\div u),
$$
for any curve $\ga$ in $\Diff(S^1)$ with $\ga'\o\ga^{-1}=u$.
\end{example}


\section{Applications: Orbit invariants in global existence results}\label{s6}

The orbit invariants described in the previous sections are powerful tools in studying the nature of the solutions of generalized Euler-Poincar\'e equations. In this section we consider equations from mathematical physics that fall into this category and prove global existence and uniqueness of solutions to the associated periodic Cauchy problems using \eqref{vval} and proposition \ref{genact}.


The four integrable equations: CH, $\mu$CH, DP and $\mu$DP from examples \ref{ex:ch}, \ref{ex:much}, \ref{ex:DP} and \ref{ex:muDP} are special cases of the equation
\begin{equation}\label{class}
\partial_t m = -um'-\lambda u'm,\quad m=\Phi u,
\end{equation}
where the operator $\Phi$ on the space of smooth functions on the circle is either a linear differential operator of the form $\sum_{j=0}^r (-1)^j \partial_x^{2j}$ or the linear operator $\mu - \partial_x^2$, 
where $\mu(u)$ is the mean of the function $u$ on $S^1$.

The equation \eqref{class} is a generalized right Euler-Poincar\'{e} equations \ref{gen} on the group of diffeomorphisms of the circle for the reduced Lagrangian 
$$
l(u)=\frac12\int_{S^1} u\Phi u dx. 
$$
In this case the $\Diff(S^1)$ action $\Theta^*$ is the action on $\lambda$--densities on the circle, with associated infinitesimal action $\th^*_uf=uf'+\la u'f$. 
The coadjoint action is obtained for $\la=2$ and in this special case \eqref{class} 
is the geodesic equation on $\Diff(S^1)$ with respect to the right invariant metric defined by the $H^r$ inner product.  

\medskip

We consider the periodic Cauchy problem for \eqref{class}:
\begin{equation}\label{eq:phi}
\partial_t u+uu'=-\Phi^{-1}\left([u,\Phi]u'+\lambda u'\Phi u\right), \quad x\in S^1, t\in\RR^+
\end{equation}
\begin{equation}\label{data:phi}
u(0,x)=u_0(x)
\end{equation}
where  $\Phi: H^s \rightarrow H^{s-r}$,  $\Phi=\sum_{j=0}^r (-1)^j \partial_x^{2j}$ and $\lambda$ is an arbitrary real number. 
The main result of this section is the following global (in time) existence and uniqueness theorem.

\begin{theorem}\label{th:global}
Let $s>2r+\frac{1}{2}$. Assume that the initial data $u_0 \in H^s(S^1)$ satisfies 
\[
\Phi u_0 \geq 0.
\]
Then the Cauchy problem (\ref{eq:phi})-(\ref{data:phi}) has a unique global solution $u$ in \\
$C(\RR^+, H^s(S^1))\cap C^1(\RR^+, H^{s-1}(S^1))$.
\end{theorem}

We postpone the proof until the end of this section and proceed to establish local well-posedness (existence, uniqueness and continuous dependence on initial data of solutions for a short time) and persistence of solutions of the Cauchy problem (\ref{eq:phi})-(\ref{data:phi}).

\begin{theorem}\label{th:local}
Let $s>2r+\frac{1}{2}$. Then the periodic Cauchy problem (\ref{eq:phi})-(\ref{data:phi}) has a unique solution 
\[ u \in C([0,T), H^s(S^1))\cap C^1 ([0,T), H^{s-1}(S^1))
 \]
for some $T>0$ and the solution depends continuously on initial data. 
\end{theorem}

Our proof of this theorem uses an approach developed by Ebin and Marsden in \cite{EM} for Euler and Navier-Stokes equations.

Let $\ga(t)$ denote the flow of $u(t)$, i.e. 
$u(t,x)=\dot{\ga}(t, \ga^{-1}(t,x))$. For convenience we use the notation 
\begin{equation}
\Psi_{\ga}\xi:=\big(\Psi(\xi\circ\ga^{-1})\big)\circ\ga
\end{equation}
for a pseudodifferential operator $\Psi$.
We write the Cauchy problem (\ref{eq:phi})-(\ref{data:phi}) as an initial value problem for an ODE in the form
\begin{equation}\label{eq:mainode}
\ddot{\ga}=-\Phi^{-1}_{\ga}\left( [\dot{\ga},\Phi_{\ga}](\partial_x)_{\ga}\dot{\ga}+\lambda \dot{\ga}'\Phi_{\ga}\dot{\ga} \right),  \quad
\dot{\ga}'(0,x)=u_0(x), \quad \ga(0,x)=x
\end{equation}
and the local well-posedness of (\ref{eq:phi})-(\ref{data:phi}) follows from Picard iterations if
\begin{equation}
F(\ga,\dot{\ga})=-\Phi^{-1}_{\ga} \left( [\dot{\ga}, \Phi_{\ga}](\partial_x)_{\ga}\dot{\ga}+\lambda \dot{\ga}' \Phi_{\ga}\dot{\ga} \right)
\end{equation}
is a continuously differentiable map from $\mathcal{D}^s \times H^s(S^1)$ into $H^s(S^1)$. Here $\mathcal{D}^s$ denotes orientation preserving circle diffeomorphisms of class $H^s$.

\begin{proof}[Proof of Theorem \ref{th:local}]
We use the symbol $\lesssim$ to denote $\leq C_{\ga}$ where $C_{\ga}$ is a constant that depends on $\lambda$ and the $H^s$ norms of $\ga$ and $\ga^{-1}$. Three important results about Sobolev spaces are used in the following estimates repeatedly: The algebra property of $H^s$ for $s>1/2$, the Sobolev imbedding theorem $C^1\hookrightarrow H^s$ for $s>3/2$ and the composition lemma (see the Appendix in \cite{BB} and lemma 4.1 in \cite{Mi}).

Our first estimate
\[ \| F(\ga, \gd)\|_{H^s}\lesssim \| [\gd\circ\gi, \Phi]\partial_x(\gd\circ\gi) \|_{H^{s-2r}}+ \| \partial_x(\gd\circ\gi) . \Phi(\gd\circ\gi)\|_{H^{s-2r}}
\]
\[ \lesssim \| \gd \circ\gi \|_{H^s}^2+\| \gd\circ\gi\|_{H^{s-2r+1}} \| \gd\circ\gi\|_{H^s}
\]
\[ \lesssim \| \gd \|_{H^s}^2
\] 
follows from composition lemma and algebra property of Sobolev spaces
and establishes that $F$ is a bounded map from ${\cal D}^s \times H^s$ into $H^s$.

The directional derivatives of $F$ are given by the formulas
\begin{eqnarray}
& \partial_{\gd}F_{(\ga,\gd)}(X)=&-\pig \big( [X,\pg]\dxg \gd + [\gd, \pg]\dxg X \big) \label{eq:Fgd1} \\
& & - \lambda \pig \big((\pg \gd)\dxg X  + (\pg X)  \dxg \gd\big) \label{eq:Fgd2}
\end{eqnarray}
and
\begin{eqnarray}
& \partial_{\ga}F_{(\ga,\gd)}(X)=&-\pig \big( [\pg,X\dxg\gd]\dxg\gd +  [\gd,\pg]\dxg(X\dxg\gd) \big)\label{eq:Fga1}\\
&&+\lambda\pig \big( (\pg\gd)X\dxg^2\gd + (\dxg\gd) \pg(X\dxg\gd)\big)\label{eq:Fga2} \\
&&-X\pig\dxg \big( [\gd,\pg]\dxg\gd + \lambda (\pg\gd)\dxg \gd \big).\label{eq:Fga3}
\end{eqnarray}
The $H^s$ norm of the term on the right hand side of \eqref{eq:Fgd1} is bounded by
\[ \lesssim \|  [X,\pg]\dxg \gd\|_{H^{s-2r}} +\|   [\gd, \pg]\dxg X\|_{H^{s-2r}}
\]
\[ \lesssim  \|X \|_{H^s} \| \gd\|_{H^s} 
\]
Similarly the term in \eqref{eq:Fgd2} is estimated by
\[ \lesssim \| \gd\|_{H^s} \| X\|_{H^{s-2r+1}}+\| X\|_{H^s} \| \gd\|_{H^{s-2r+1}}
\]
\[ \lesssim \| \gd\|_{H^s} \| X\|_{H^s}.
\]
Therefore we have
\begin{equation} \label{eq:est2}
\| \partial_{\gd}F_{(\ga,\gd)}(X)\|_{H^s} \lesssim \|  \gd\|_{H^s} \|X\|_{H^s}.
\end{equation}
In order to estimate the $H^s$ norm of $\partial_{\ga}F_{(\ga,\gd)}(X)$ we rearrange the terms in \eqref{eq:Fga1}-\eqref{eq:Fga3}:
\begin{eqnarray}
& \partial_{\ga}F_{(\ga,\gd)}(X)=&[X,\pig]\big( (\pg\gd)\dxg^2 \gd\big) \label{eq:Fga21}\\
&&+\pig [\gd,\pg]\big( (\dxg X)\dxg\gd \big)\label{eq:Fga22} \\
&&-[X,\pig[\gd,\pg]]\dxg^2\gd \label{eq:Fga23} \\
&&-(\lambda+1) [X,\pig]\dxg\big( (\pg\gd)\dxg \gd \big) \label{eq:Fga24} \\
&&+\lambda \pig\big( (\dxg\gd)[\pg,X]\dxg\gd\big) \label{eq:Fga25}.
\end{eqnarray}
For the term in \eqref{eq:Fga21} we have
\[
\| [X,\pig]\big( (\pg\gd)\dxg^2 \gd\big)\|_{H^s} \leq \| X\|_{H^s} \| (\pg \gd)\dxg^2 \gd\|_{H^{s-2r}}
\]
\[
\lesssim \| X\|_{H^s} \| \gd\|_{H^s}^2.
\]
The estimate for \eqref{eq:Fga22} is
\[
\| \pig [\gd,\pg]\big( (\dxg X)\dxg\gd \big)\|_{H^s} \lesssim \|[\gd,\pg]\big( (\dxg X)\dxg\gd \big)\|_{H^{s-2r}}
\]
\[
\lesssim \| \gd\|_{H^s} \| (\dxg X)\dxg \gd\|_{H^{s-1}}
\]
\[
\lesssim \|\gd \|_{H^s}^2 \|X \|_{H^s}.
\]
The nested commutators in \eqref{eq:Fga23} helps us bound its $H^s$ norm:
\[
\| [X,\pig[\gd,\pg]]\dxg^2\gd\|_{H^s} \lesssim \| X\|_{H^s} \| \gd\|_{H^s}\|\dxg^2 \gd \|_{H^{s-2}}
\]
\[
\lesssim \| X\|_{H^s}\| \gd\|_{H^s}^2.
\]
The $H^s$ norms of the last two terms \eqref{eq:Fga24} and \eqref{eq:Fga25} are estimated similarly to give 
\begin{equation}\label{eq:est3}
\| \partial_{\ga}F_{(\ga,\gd)}(X)\|_{H^s} \lesssim \|  \gd\|_{H^s}^2 \|X\|_{H^s}.
\end{equation}
Therefore both $\partial_{\gd}F_{(\ga,\gd)}$ and $\partial_{\ga}F_{(\ga,\gd)}$ are bounded linear operators on the space of $H^s$ functions.

In order to complete the proof of the theorem, it is sufficient to establish that $\partial_{\gd}F_{(\ga,\gd)}$ and $\partial_{\ga}F_{(\ga,\gd)}$ depend continuously on $(\ga,\gd)$ in some neighbourhood of $(\mbox{id},0)$ in ${\cal D}^s\times H^s$. 
Note that $\| \partial_{\ga}F_{(\ga,\gd)}(X)-\partial_{\ga}F_{(\mbox{id},\gd)}(X)\|_{H^s}$ is a sum of differences corresponding to each term in \eqref{eq:Fga21}-\eqref{eq:Fga25}. For instance the difference corresponding to \eqref{eq:Fga21} is
\begin{equation}\label{eq:sample}
[X,\pig] \big( (\pg\gd)\dxg^2 \gd\big)-[X,\Phi^{-1}]\big( (\Phi\gd)\dx^2 \gd\big).
\end{equation}
We add and subtract appropriate terms to bound the $H^s$ norm of \eqref{eq:sample}:
\begin{equation}
 \| [X,\pig] \big( (\pg\gd)\dxg^2 \gd\big)-[X,\Phi^{-1}]\big( (\Phi\gd)\dx^2 \gd\big)\|_{H^s}
\end{equation}
\begin{equation}\label{eq:sam1}
\leq \| [X,\pig] \big( (\pg\gd)\dxg^2 \gd\big)-[X,\pig] \big( (\pg\gd)\dxg^2 \gd\big)\circ\gi\|_{H^s}
\end{equation}
\begin{equation}\label{eq:sam2}
 +\|[X\circ\gi,\Phi^{-1}]\big(\Phi(\gd\circ\gi).\dx^2(\gd\circ\gi) -(\Phi\gd)\dx^2 \gd\big) \|_{H^s}
\end{equation}
\begin{equation}\label{eq:sam3}
 +\|[X\circ\gi-X,\Phi^{-1}]\big( (\Phi\gd)\dx^2 \gd\big) \|_{H^s}.
\end{equation}
Here we use another property of composition of $H^s$ functions with $H^s$ class diffeomorphisms, lemma 4.2 in \cite{Mi} (see \cite{BB} for the proof in the case of $s$ integer), to estimate all three terms \eqref{eq:sam1}-\eqref{eq:sam3}. For \eqref{eq:sam1} we have
\begin{equation}
\| [X,\pig] \big( (\pg\gd)\dxg^2 \gd\big)-[X,\pig] \big( (\pg\gd)\dxg^2 \gd\big)\circ\gi\|_{H^s}
\end{equation}
\begin{equation}
\lesssim \| [X\circ\gi,\Phi^{-1}]\big(\Phi(\gd\circ\gi)\|_C^2 \|\ga -\mbox{id}\|_{H^s}
\end{equation}
\begin{equation}
\lesssim \|X\|_{H^s} \|\gd\|_{H^s}\|\gd\|_{H^{s-2r+2}} \|\ga -\mbox{id}\|_{H^s}.
\end{equation}
The term in \eqref{eq:sam2} is bounded by
\[
\|X\|_{H^s} \| \Phi(\gd\circ\gi).\dx^2(\gd\circ\gi) -(\Phi\gd)\dx^2 \gd\|_{H^{s-2r}}
\]
\[
\lesssim \|X\|_{H^s} \| \Phi(\gd\circ\gi)\big(\dx^2(\gd\circ\gi) -\dx^2 \gd\big)\|_{H^{s-2r}}
\]
\[
+\|X\|_{H^s} \| \big(\Phi(\gd\circ\gi)-\Phi\gd\big)\dx^2 \gd\|_{H^{s-2r}}
\]
\[
\lesssim \|X\|_{H^s} \| \gd\|_{H^s}^2 \|\ga-\mbox{id}\|_{H^s}.
\]
The estimate on \eqref{eq:sam3} is given by
\begin{equation}
\|[X\circ\gi-X,\Phi^{-1}]\big( (\Phi\gd)\dx^2 \gd\big) \|_{H^s}
\end{equation}
\begin{equation}
\lesssim \|X\circ\gi-X\|_{H^s} \|\gd\|_{H^s} \|\gd \|_{H^{s-2r+2}}
\end{equation}
\begin{equation}
\lesssim \|X\|_{H^s} \|\ga -\mbox{id}\|_{H^s} \|\gd \|_{H^s}^2.
\end{equation}
For all the difference terms corresponding to \eqref{eq:Fga22}-\eqref{eq:Fga25} are bounded similarly
by
\(
\|X\|_{H^s} \|\ga -\mbox{id}\|_{H^s} \|\gd \|_{H^s}^2
\), 
hence we have
\begin{equation}\label{eq:contga}
\| \partial_{\ga}F_{(\ga,\gd)}(X)-\partial_{\ga}F_{(\mbox{id},\gd)}(X)\|_{H^s} \lesssim \|X\|_{H^s} \|\ga -\mbox{id}\|_{H^s} \|\gd \|_{H^s}^2.
\end{equation}
Furthermore, for $\partial_{\gd}F_{(\ga,\gd)}$, the estimate
\begin{equation}\label{eq:contgd}
\| \partial_{\gd}F_{(\ga,\gd)}(X)-\partial_{\ga}F_{(\mbox{id},\gd)}(X)\|_{H^s} \lesssim \|X\|_{H^s} \|\ga -\mbox{id}\|_{H^s} \|\gd \|_{H^s}^2
\end{equation}
follows using the same techniques.

The inequalities \eqref{eq:contga} and \eqref{eq:contgd} with the observation that both directional derivatives $\partial_{\ga}F_{(\ga,\gd)}$ and $\partial_{\gd}F_{(\ga,\gd)}$ are linear in $\gd$ imply that $(\gd, F(\ga,\gd))$ defines a continuously differentiable vector field in a neighbourhood of $(\mbox{id}, 0)$ in the space ${\cal D}^s\times H^s $. Therefore the classical Picard iterations apply to the Cauchy problem \eqref{eq:mainode}.
\end{proof}

For $\lambda=2$, equation (\ref{eq:phi}) is an equation for geodesics for the right invariant metric induced by the $H^r$ inner product. The local well-posedness is proved in \cite{CKKT} in this case.

Another case of interest is when $\Phi=\mu-\partial_x^2$ where $\mu(u)=\int_{S^1}u(x)dx$. Both local well-posedness and global existence results of theorems \ref{th:local} and \ref{th:global} are shown in \cite{LMT} for $\lambda>0$ in this case.

The following proposition is a persistence result for Sobolev class solutions of (\ref{eq:phi})-(\ref{data:phi}); i.e. it provides a condition under which the short time solutions persist for all time. It is in the spirit of the persistence result of Beale, Kato and Majda in \cite{BKM} for Euler equations of hydrodynamics.

\begin{proposition}\label{prop:per}
Let $s>2r+\frac{1}{2}$ and let $u\in C([0,T),H^s(S^1))$ be a solution of (\ref{eq:phi})-(\ref{data:phi}). If there exists a $K>0$ such that
\[
\| u(t) \|_{C^1}\leq K<\infty
\]
for all $t$ then $u$ can be extended to a solution of (\ref{eq:phi})-(\ref{data:phi}) that exists for all time.
\end{proposition}

\begin{proof}
Using Friedrich's mollifiers $J_{\ep}$ we have
\begin{eqnarray}
&\frac{d}{dt}\| J_{\ep}u\|_{H^s} & = \langle 2\Lambda^s \partial_t J_{\ep} u, \Lambda^s  J_{\ep} u \rangle_{L^2} \nonumber \\
& & = -2 \langle \Lambda^s  J_{\ep} (uu'), \Lambda^s  J_{\ep} u \rangle_{L^2} \label{eq:mol1}\\
& & \ \ \ -2 \langle \Lambda^s \Phi^{-1} ([\Phi, J_{\ep}u]J_{\ep}u' , \Lambda^s  J_{\ep} u \rangle_{L^2} \label{eq:mol2}\\
& & \ \ \ -2\lambda \langle \Lambda^s  \Phi^{-1}(J_{\ep}u' \Phi J_{\ep}u), \Lambda^s  J_{\ep} u \rangle_{L^2} \label{eq:mol3} 
\end{eqnarray}
The mollifiers are used in estimating the term (\ref{eq:mol1})
\begin{equation} \label{eq:mol11}
\langle \Lambda^s J_{\ep}(uu'), \Lambda^s J_{\ep}u\rangle_{L^2} \lesssim \|u\|_{C^1} \|u\|_{H^s}^2
\end{equation}
This is proved in \cite{Taylor}. 
We use Cauchy-Schwarz to bound the second term (\ref{eq:mol2}) by
\[
 \| [J_{\ep}u, L ]J_{\ep}u'\|_{H^{s-2r}} \| J_{\ep}u\|_{H^s}
\]
By simply observing that the highest derivative on $u$ in the commutater $[J_{\ep}u, L]$ is $\partial_x^{2r-1}$ we obtain the estimate
\begin{equation}\label{eq:mol22}
\langle \Lambda^s \Phi^{-1} ([\Phi, J_{\ep}u]J_{\ep}u' , \Lambda^s  J_{\ep} u \rangle_{L^2} \lesssim \| u\|_{H^{s-1}} \| u\|_{H^s}^2
\end{equation}
On the third term (\ref{eq:mol3}) we use once again Cauchy-Schwarz to obtain
\[
\langle \Lambda^s  \Phi^{-1}(J_{\ep}u' \Phi J_{\ep}u), \Lambda^s  J_{\ep} u \rangle_{L^2} \lesssim \| J_{\ep}u' \Phi J_{\ep} u\|_{H^{s-2r}} \|J_{\ep} u \|_{H^s}
\]
By our assumption on the indices that $s-2r>1/2$ and $r\geq1$ we have
\begin{equation}\label{eq:mol33}
 \lambda \langle \Lambda^s  \Phi^{-1}(J_{\ep}u' \Phi J_{\ep}u), \Lambda^s  J_{\ep} u \rangle_{L^2} \lesssim \| u \|_{H^{s-1}} \| u \|_{H^s}^2
\end{equation}
Putting all three estimates (\ref{eq:mol11}), (\ref{eq:mol22}) and (\ref{eq:mol33}) together we obtain
\begin{equation}
 \frac{d}{dt} \| J_{\ep}u\|_{H^s}^2 \lesssim \|u\|_{C^1} \|u\|_{H^s}^2
\end{equation}
Passing to the limit as $\ep \rightarrow 0$ and using Gronwall's inequality we obtain the persistence result of proposition \ref{prop:per}.
\end{proof}

Now we have all the ingredients for the proof of the global existence and uniqueness result.

\begin{proof}[Proof of Theorem \ref{th:global}]
 The Sobolev embedding theorem  implies
\begin{equation}
\|\partial_x u \|_{\infty} \lesssim \|\Phi u\|_{L^1}.
\end{equation}
The orbit invariant mentioned in proposition \ref{genact}  guarantees that $\Phi u_0 \geq 0$ implies $\Phi u\geq 0$. Furthermore the integral $\int_{S^1} \Phi u dx$ is conserved, hence we have
\begin{equation}
\| \Phi u \|_{L^1}=\int_{S^1}\Phi u dx = \int_{S^1}\Phi u_0 dx.
\end{equation}
Therefore, by proposition \ref{prop:per}, the solution of theorem \ref{th:local} persists for all time.
\end{proof}



Section de Math\'ematiques, {\'E}cole Polytechnique F{\'e}d{\'e}rale de Lausanne, CH--1015 Lausanne, Switzerland, {\it \small e-mail: feride.tiglay@epfl.ch}

Department of Mathematics, West University of Timi\c soara,
Bd.~V.~P\^arvan 4, 300223-Timi\c soara, Romania,
{\it \small e-mail: vizman@math.uvt.ro}

\end{document}